\documentclass[notitlepage,a4paper,twoside,fleqn,11pt,jmp, leqno]{revtex4-1}

\usepackage{hyperref}
\hypersetup{colorlinks=true,allcolors=[rgb]{0,0,0.6}}
\usepackage[utf8]{inputenc}
\usepackage{color}
\usepackage{calc}
\usepackage{newlfont}
\usepackage[english]{babel}
\usepackage{amssymb,amsmath,amsfonts,amsthm}
\usepackage{mathrsfs}
\usepackage{dsfont}
\usepackage{enumerate}
\usepackage{tgschola}
\usepackage[T1]{fontenc}
\renewcommand{\vec}[1]{\mathbf{#1}}
\newcommand{\de}{\mathrm{d}}

\newcommand{\braket}[2]{ \langle #1 , #2 \rangle }  
\newcommand{\norm}[1]{\bigl\lVert #1 \bigr\rVert} 
\newcommand{\ass}[1]{\left\lvert #1 \right\rvert} 
\newcommand{\dtot}[2]{\frac{\mathrm{d} #1}{\mathrm{d} #2}} 
\newcommand{\ide}[1]{\mathrm{d}#1 \,} 

\renewcommand{\v}{\widetilde{v}}

\newcommand{\F}{\mathscr{F}}

\theoremstyle{plain}
\newtheorem{thm}{Theorem}

\newtheorem*{thm1*}{Theorem 1$^{\mathbf{(r)}}$}
\newtheorem*{thm2*}{Theorem 2$^{\mathbf{(r)}}$}
\newtheorem*{assertion*}{Assertion}
\newtheorem{proposition}{Proposition}

\newtheorem{lemma}{Lemma}
\newtheorem*{lemma*}{Lemma}
\newtheorem{corollary}{Corollary}
\newtheorem*{corollary*}{Corollary}
\theoremstyle{definition}

\newtheorem*{definition*}{Definition}
\newtheorem*{definitions*}{Definitions}
\newtheorem{innercondition}{Condition}

\theoremstyle{remark}

\newtheorem{remark}{Remark}

\newtheorem*{remark*}{Remark}
\newtheorem*{remarks*}{Remarks}
\renewcommand\thesection{\arabic{section}}

\makeatletter
\renewcommand\p@subsection{\thesection .}
\makeatother

\begin{document}
\date{\today}
\title{Global Solution of the Electromagnetic Field-Particle System of Equations.}
\author{Marco Falconi}
\thanks{Current affiliation: IRMAR, Université de Rennes I. 263 Av. Général Leclerc CS 74205, 35042 Rennes.}
\affiliation{Dipartimento di Matematica, Università di Bologna\\Piazza di Porta San Donato 5 - 40126 Bologna, Italy}
\email[]{m.falconi@unibo.it}
\begin{abstract}
In this paper we discuss global existence of the solution of the Maxwell and Newton system of equations, describing the interaction of a rigid charge distribution with the electromagnetic field it generates. A unique solution is proved to exist (for regular charge distributions) on suitable homogeneous and non-homogeneous Sobolev spaces, for the electromagnetic field, and on coordinate and velocity space for the charge; provided initial data belong to the subspace that satisfies the divergence part of Maxwell's equations.
\end{abstract}
\maketitle

\section{Introduction.}
\label{sec:introduction}

We are interested in the following system of equations: let $e\in\mathds{R}$, $\varphi$ a sufficiently regular function; then the Maxwell-Newton equations are written in three spatial dimensions as
\begin{equation}
  \label{eq:1}\tag{M-N}
    \begin{aligned}
      &\left\{\begin{aligned}
       \partial_t &B + \nabla\times E=0\\
       \partial_t &E - \nabla\times B=-j
      \end{aligned}\right. \mspace{20mu}
      \left\{\begin{aligned}
        \nabla\cdot &E=\rho\\
        \nabla\cdot &B=0
      \end{aligned}\right.\\
      &\left\{\begin{aligned}
        \dot{\xi}&= v\\
        \dot{v}&= e[(\varphi*E)(\xi)+v\times(\varphi*B)(\xi)]
      \end{aligned}\right.
    \end{aligned}
\qquad ;
\end{equation}
with
\begin{equation}\label{eq:13}
  j=ev\varphi(\xi-x)\; ,\quad  \rho=e\varphi(\xi-x)\; .
\end{equation}
This system can be used to describe motion of a non-relativistic rigid particle, with an extended charge distribution $e\varphi$, interacting with its own electromagnetic field (in this case we could need some additional physical conditions, such as $\int\de x \varphi=1$, however these conditions are not necessary for the existence of a solution). So $\xi,v\in\mathds{R}^3$ will be the position and velocity of the charge's center of mass, $E,B$ the electric and magnetic field vectors. We remark that in \eqref{eq:1} charge is conserved, i.e.
\begin{equation}
  \label{eq:2}
  \partial_t\rho+\nabla\cdot j=0\; .
\end{equation}
It is useful to construct the electromagnetic tensor $F^{\mu\nu}$:
\begin{equation*}
  F^{\mu\nu}=\left(
  \begin{array}{cccc}
    0&E_1&E_2&E_3\\
   -E_1&0&B_3&-B_2\\
   -E_2&-B_3&0&B_1\\
   -E_3&B_2&-B_1&0
  \end{array}\right)\; .
\end{equation*}
Therefore we make the following identifications:
\begin{gather*}
  E_j=\sum_{j=1}^3\delta_{ij}F^{0j}\; ,\\
B_j=\frac{1}{2}\sum_{k,l=1}^3\epsilon_{jkl}F^{kl}\; ;
\end{gather*}
where $\delta_{ij}$ is the Kronecker's delta and $\epsilon_{ijk}$ is the three-dimensional Levi-Civita symbol. From now on, we adopt the following notation: whenever an index is repeated twice, a summation over all possible values of such index is intended. Define $R^j_{kl}=-\delta_l^j\partial_k+\delta_k^j\partial_l$ and $(R^*)^{kl}_j=\delta_j^l\partial^k-\delta^k_j\partial^l$, and let $\Omega=(RR^*)^{1/2}$, then
\begin{equation}
  \label{eq:4}
  U(t)\equiv\left(
    \begin{array}{cc}
      \;\cos\Omega t\;&\;\frac{\sin\Omega t}{\Omega}R\;\\
      \; -R^*\frac{\sin\Omega t}{\Omega}\; &\; \cos(R^*R)^{1/2}t\;
    \end{array}
\right)\; ;
\end{equation}
also, define
\begin{equation}
  \label{eq:5}
  W(t)\equiv 1+t\left(\begin{array}{cc}
    \;0\;&\;1\;\\\;0\;&\;0\;
  \end{array}\right )\; .
\end{equation}
For the construction of $U(t)$ we have followed \citep{zbMATH03764412}. Then \eqref{eq:1} can be rewritten as an integral equation: set
\begin{equation}\label{eq:3}
  \vec{u}(t)=\left(
    \begin{array}{c}
      F_0(t)\\F(t)\\\xi(t)\\v(t)
    \end{array}
  \right)\; ,
\end{equation}
and let $\vec{u}(t_0)=\vec{u}_{(0)}$ (we define the electric part as $F_0=F^{0j}$, the magnetic part as $F=(F^{jk})_{j<k}$); also, let
\begin{gather*}
  j(t)=ev(t)\varphi(\xi(t)-x)\; ,\\
(f_{em}(t))_i=e\sum_{j=1}^3\delta_{ij}\bigl[\bigl(\varphi*F^{0j}(t)\bigr)(\xi(t))+\sum_{k=1}^3 v_k(t)\bigl(\varphi*F^{jk}(t)\bigr)(\xi(t))\bigr]\; .
\end{gather*}
Then we can write the integral equation:
\begin{equation}
  \label{eq:6}\tag{M-N.i}
  \vec{u}(t)=\left(
    \begin{array}{cc}
      U(t-t_0)&0\\0&W(t-t_0)
    \end{array}
\right)\vec{u}_{(0)}+\int_{t_0}^t\de\tau\left(
    \begin{array}{cc}
      U(t-\tau)&0\\0& W(t-\tau)
    \end{array}
\right) \left(
    \begin{array}{c}
      -j(\tau)\\0\\0\\f_{em}(\tau)
    \end{array}
\right)\; .
\end{equation}
The second couple of Maxwell's equations (namely $\nabla\cdot E=\rho$ and $\nabla\cdot B=0$) have to be dealt with separately.

We are interested in solutions belonging to the following spaces: let $\mathscr{X}_{s}$, for $-\infty<s<3/2$, to be
\begin{equation*}
  \mathscr{X}_{s}\equiv \bigl(\dot{H}^s(\mathds{R}^3) \bigr)^3\oplus \bigl(\dot{H}^s(\mathds{R}^3) \bigr)^3\oplus \mathds{R}^3\oplus\mathds{R}^3\; .
\end{equation*}
If $\vec{u}\in\mathscr{X}_{s}$ has the form \eqref{eq:3}, then $\mathscr{X}_{s}$ is a Hilbert space if equipped with the norm
\begin{equation*}
  \norm{\vec{u}}_{\mathscr{X}_{s}}^2=\norm{F_0}^2_{(\dot{H}^s)^3}+\norm{F}^2_{(\dot{H}^s)^3}+\ass{\xi}^2+\ass{v}^2\; ;
\end{equation*}
where
\begin{equation*}
  \norm{f}^2_{(\dot{H}^s)^3}=\sum_{j=1}^3\norm{\omega^s f_j}^2_{L^2}=\sum_{j=1}^3\int_{\mathds{R}^3}\ide{k}\ass{k}^{2s}\lvert\hat{f_j}(k)\rvert^2\; ,
\end{equation*}
with
\begin{equation*}
  \omega=\ass{\nabla}\; .
\end{equation*}
The homogeneous Sobolev spaces $\dot{H}^s(\mathds{R}^d)$ are Hilbert spaces for all $s<d/2$ \citep[see][]{MR2768550}. We will use as well the non-homogeneous Sobolev spaces $H^r(\mathds{R}^d)$, $r\geq 0$, complete with the norm
\begin{equation*}
  \norm{f}^2_{H^r}=\int_{\mathds{R}^d}\ide{k}(1+\ass{k}^2)^r\lvert\hat{f}(k)\rvert^2\; .
\end{equation*}
Let $I\subseteq \mathds{R}$; then we define
\begin{equation*}
  \mathscr{X}_{s}(I)=\mathscr{C}^0(I,\mathscr{X}_{s})\; .
\end{equation*}
If $I$ is compact, $\mathscr{X}_s(I)$ is complete with the norm
\begin{equation*}
  \norm{\vec{u}(\cdot)}_{\mathscr{X}_{s}(I)}=\sup_{t\in I}\norm{\vec{u}(t)}_{\mathscr{X}_{s}}\; .
\end{equation*}
Our goal will be to prove that a unique global solution of \eqref{eq:1} exists on $\mathscr{X}_{s}(\mathds{R})$ whenever the initial datum belongs to (a subspace of) $\mathscr{X}_{s}$ (theorem~\ref{sec:stat-main-results-1}); for all $s<3/2$ and suitably regular $\varphi$ (by the result for $s=0$ global existence on non-homogenous Sobolev spaces is also proved, theorem~\ref{sec:stat-main-results-2}).

Particles interacting with its electromagnetic field has been widely
studied in physics. The study of a radiating point particle revealed
the presence of divergencies. Therefore classically a radiating
particle is always assumed to have extended charge distribution. The
most used equations to describe such particle's (and corresponding
fields) motion are~\eqref{eq:1} above, or its semi-relativistic
counterpart, called Abraham model \citep{MR2097788} (also
Maxwell-Lorentz equations in \citep{bauerdurr2001}), see~\eqref{eq:14}
below. For a detailed discussion of their physical properties,
historical background and applications the reader can consult any
classical textbook on electromagnetism
\citep[e.g.][]{rohrlich2007classical}. On a mathematical standpoint,
almost all results deal with the semi-relativistic system, and are
quite recent. We mention an early work of \citet{bambusi.noja} on the
linearised problem; papers of \citet{Kiessling1999197,Appel200124} on
conservation laws and motion of a rotating extended charge. Concerning
global existence of solutions, refer to
\citet{doi:10.1080/03605300008821524} and \citet{bauerdurr2001}; the
latter result has been developed further in
\citet{MR3169754,MR3085905} to consider weighted $L^2$
spaces. \citet{mauser.imaikin,imaykin:042701} have investigated
soliton-type solutions and asymptotics. For a comprehensive review on
the classical and quantum dynamics of particles and their radiation
fields the reader may refer to the book by \citet{MR2097788}.

The existence results \citep{doi:10.1080/03605300008821524,bauerdurr2001,MR3169754} are formulated for the semi-relativistic system~\eqref{eq:14}, but they should apply also to \eqref{eq:1}: existence of a differentiable solution holds on suitable subspaces of $(L^2_w(\mathds{R}^3))^3\oplus (L^2_w(\mathds{R}^3))^3\oplus \mathds{R}^3\oplus \mathds{R}^3$, with $w$ denoting an eventual weight on $L^2$ spaces; in this paper a continuous global solution is proved to exist on a wider class of spaces: $\bigl(\dot{H}^s(\mathds{R}^3) \bigr)^3\oplus \bigl(\dot{H}^s(\mathds{R}^3) \bigr)^3\oplus \mathds{R}^3\oplus\mathds{R}^3$, $s<3/2$ and $\bigl(H^r(\mathds{R}^3) \bigr)^3\oplus \bigl(H^r(\mathds{R}^3) \bigr)^3\oplus \mathds{R}^3\oplus\mathds{R}^3$, $r\geq 0$.

From a physical standpoint, it is a natural choice to consider the
energy space $\mathscr{X}_0(\mathds{R})$ to solve \eqref{eq:1}. Also,
it is not necessary, in principle, to introduce new objects like the
electromagnetic potentials $\phi $ and $A$. Nevertheless, there are
plenty of situations where such potentials are either convenient or
necessary: a simple example is in defining a Lagrangian or Hamiltonian
function of the charge-electromagnetic field system. Once a gauge is
fixed, investigating the regularity of the potentials $\phi $ and $A$
is equivalent to provide a solution to \eqref{eq:1} on a suitable
space, often different form $\mathscr{X}_0(\mathds{R})$. In this
context, homogeneous Sobolev spaces emerge. For example, consider the
vector potential $A$ in the Coulomb gauge. Then, given $F^{ij}$, we
have $A_{j}=\omega ^{-2}\sum_{i=1}^{3}\partial _i F^{ij}$. So the
requirement $A\in \bigl(L^2 (\mathds{R}^3 )\bigr)^3$ is equivalent to
$B\in \bigl(\dot{H}^{-1}(\mathds{R}^3)\bigr)^{3}$. Another natural
choice, if one wants to investigate the connection between the quantum
and the classical theory, is to have $A\in
\bigl(\dot{H}^{1/2}(\mathds{R}^3)\bigr)^{3}$ and $\partial _t A\in
\bigl(\dot{H}^{-1/2}(\mathds{R}^3)\bigr)^{3}$ (because, roughly
speaking, the classical correspondents of quantum
creation/annihilation operators behave as $\omega ^{1/2}A$, $\omega
^{-1/2}\partial _t A$ and are required to be square integrable). This
is equivalent to solve \eqref{eq:1} with $E\in
\bigl(\dot{H}^{-1/2}(\mathds{R}^3)\bigr)^{3}$ and $B\in
\bigl(\dot{H}^{-1/2}(\mathds{R}^3)\bigr)^{3}$. This led us to consider
the existence of global solutions of \eqref{eq:1} on homogeneous
Sobolev spaces, especially the ones with negative index $s<0$.
\begin{remark}
  In formulating the equations, we have restricted to one charge for
  the sake of simplicity; results analogous to those stated in
  Theorems~\ref{sec:stat-main-results-1} and
  \ref{sec:stat-main-results-2} should hold also in the case of $n$
  charges, even if they are subjected to mutual and external
  interactions, provided these interactions are regular enough.
\end{remark}

The rest of the paper is organised as follows: in section~\ref{sec:stat-main-results} we summarise and discuss the results proved in this paper; in section~\ref{sec:priori-cons} a local solution is constructed by means of Banach fixed point theorem; in section~\ref{sec:uniq-maxim-solut} we prove uniqueness and construct the maximal solution; in section~\ref{sec:global-existence} we show that the maximal solution is defined for all $t\in\mathds{R}$; finally in section~\ref{sec:second-couple-maxw} we discuss the divergence part of Maxwell's equations.

\section{Statement of main results.}
\label{sec:stat-main-results}

In this section we summarise the results proved in the paper. This is done in Theorems~\ref{sec:stat-main-results-1} and~\ref{sec:stat-main-results-2}. We recall the Cauchy problem related to the Maxwell-Newton system of equations:
\begin{equation}
  \label{eq:11}\begin{aligned}
      \left\{\begin{aligned}
       \partial_t &B + \nabla\times E=0\\
       \partial_t &E - \nabla\times B=-j
      \end{aligned}\right. \mspace{20mu}&\left\{\begin{aligned}
        \dot{\xi}&= v\\
        \dot{v}&= e[(\varphi*E)(\xi)+v\times(\varphi*B)(\xi)]
      \end{aligned}\right.\\
      \left\{\begin{aligned}
       E(t_0)&=E_{(0)}\\
       B(t_0)&=B_{(0)}
      \end{aligned}\right. \mspace{69mu}&\left\{\begin{aligned}
        \xi(t_0)&= \xi_{(0)}\\
        v(t_0)&=v_{(0)}
      \end{aligned}\right.
\end{aligned}\quad ;
\end{equation}
\begin{equation}
  \label{eq:12}
\left\{\begin{aligned}
        \nabla\cdot &E=\rho\\
        \nabla\cdot &B=0
      \end{aligned}\right. \; ;
\end{equation}
with $j=ev\varphi(\xi-x)$ and $\rho=e\varphi(\xi-x)$. With an abuse of terminology, we will refer to the solutions of the Cauchy problem~\eqref{eq:11}, \eqref{eq:12} simply as the solutions of Maxwell-Newton system.
\begin{thm}\label{sec:stat-main-results-1}
Let $-\infty<s<3/2$, $\varphi$ a differentiable function of $\mathds{R}^3$ such that $\norm{\varphi}_Y<\infty$ (~$Y$ is $\dot{H}^{-s}$, $\dot{H}^{-s+1}$, $\dot{H}^{s}$, $\dot{H}^{s+1}$) and such that~\eqref{eq:12} admits a solution on $(\dot{H}^s(\mathds{R}^3))^3\oplus (\dot{H}^s(\mathds{R}^3))^3$. Furthermore let $\xi_{(0)},v_{(0)}\in\mathds{R}^3$ and $E_{(0)}, B_{(0)}\in (\dot{H}^s(\mathds{R}^3))^3$ satisfying~\eqref{eq:12}.

Then the Maxwell-Newton system~\eqref{eq:1} admits a unique solution on

$\mathscr{C}^0(\mathds{R},(\dot{H}^s(\mathds{R}^3))^3\oplus (\dot{H}^s(\mathds{R}^3))^3\oplus\mathds{R}^3\oplus\mathds{R}^3)$.
\end{thm}
\begin{corollary}
  Let the conditions of theorem~\ref{sec:stat-main-results-1} be satisfied. In addition, let $\norm{\varphi}_{\dot{H}^{s-1}}<\infty$, $E_{(0)}, B_{(0)}\in (\dot{H}^{s-1}(\mathds{R}^3))^3$.

Then~\eqref{eq:1} admits a unique solution on $\mathscr{C}^0(\mathds{R},(\dot{H}^s(\mathds{R}^3))^3\oplus (\dot{H}^s(\mathds{R}^3))^3\oplus\mathds{R}^3\oplus\mathds{R}^3)$ $\cap$ $\mathscr{C}^1(\mathds{R},(\dot{H}^{s-1}$

$(\mathds{R}^3))^3\oplus (\dot{H}^{s-1}(\mathds{R}^3))^3\oplus\mathds{R}^3\oplus\mathds{R}^3)$.
\end{corollary}
\begin{thm}\label{sec:stat-main-results-2}
  Let $r\geq 0$, $\varphi$ a differentiable function of $\mathds{R}^3$ such that $\norm{\varphi}_{H^{r}}$, $\norm{\varphi}_{H^1}<\infty$ and such that~\eqref{eq:12} admits a solution on $(H^r(\mathds{R}^3))^3\oplus (H^r(\mathds{R}^3))^3$. Furthermore let $\xi_{(0)},v_{(0)}\in\mathds{R}^3$ and $E_{(0)}, B_{(0)}\in (H^r(\mathds{R}^3))^3$ satisfying~\eqref{eq:12}.

Then~\eqref{eq:1} admits a unique solution on $\mathscr{C}^0(\mathds{R},(H^r(\mathds{R}^3))^3\oplus (H^r(\mathds{R}^3))^3\oplus\mathds{R}^3\oplus\mathds{R}^3)\cap \mathscr{C}^1(\mathds{R},(H^{r-1}$

$(\mathds{R}^3))^3\oplus (H^{r-1}(\mathds{R}^3))^3\oplus\mathds{R}^3\oplus\mathds{R}^3)$.
\end{thm}
\begin{proof}[Proof of Theorem~\ref{sec:stat-main-results-2}]
  Using Theorem~\ref{sec:stat-main-results-1} we prove existence of the unique solution of~\eqref{eq:1} on $\mathscr{X}_0(\mathds{R})$. Then by Remark~\ref{sec:global-existence-2} it follows that if initial data are in $(H^r(\mathds{R}^3))^3\oplus (H^r(\mathds{R}^3))^3\oplus\mathds{R}^3\oplus\mathds{R}^3$ then also the solution is in the same space for all $t\in\mathds{R}$.
\end{proof}
\begin{remark}\label{sec:stat-main-results-3}
  The methods and results of this paper should also apply to the semi-relativistic version of Maxwell-Newton system, called Abraham model:
  \begin{equation}\label{eq:14}
    \begin{aligned}
      &\left\{\begin{aligned}
       \partial_t &B + \nabla\times E=0\\
       \partial_t &E - \nabla\times B=-j
      \end{aligned}\right. \mspace{20mu}
      \left\{\begin{aligned}
        \nabla\cdot &E=\rho\\
        \nabla\cdot &B=0
      \end{aligned}\right.\\
      &\left\{\begin{aligned}
        \dot{\xi}&= \frac{p}{\sqrt{1+p^2}}\\
        \dot{p}&= e[(\varphi*E)(\xi)+\frac{p}{\sqrt{1+p^2}}\times(\varphi*B)(\xi)]
      \end{aligned}\right.
    \end{aligned}    \quad ;
  \end{equation}
with $j=ep\varphi(\xi-x)/\sqrt{1+p^2}$ and $\rho=e\varphi(\xi-x)$.

In \citet{bauerdurr2001} an approach different to the one of this paper is taken; global existence for~\eqref{eq:14} is proved on a particular class of $(L^2(\mathds{R}^3))^3\oplus (L^2(\mathds{R}^3))^3\oplus \mathds{R}^3\oplus \mathds{R}^3$-subspaces, namely $D(B^n)$, $n\geq 1$ where $B\vec{u}=(\nabla\times B,-\nabla\times E,0,0)$.
\end{remark}

\subsection{Rotating charge.}
\label{sec:rotating-charge}

Let $m,I>0$, and $\Omega(\cdot):\mathds{R}\to\mathds{R}^3$. Then global solution of the following Cauchy problem can be proved along the same guidelines as for theorems~\ref{sec:stat-main-results-1},~\ref{sec:stat-main-results-2}:
\begin{equation}
  \label{eq:15}
  \begin{aligned}
      &\left\{\begin{aligned}
       \partial_t &B + \nabla\times E=0\\
       \partial_t &E - \nabla\times B=-j
      \end{aligned}\right.\\ 
      &\left\{\begin{aligned}
        \dot{\xi}&= v\\
        \dot{v}&= \frac{e}{m}\int\ide{x}[E(x)+(v+\Omega\times(x-\xi))\times B(x)]\varphi(\xi-x)\\
        \dot{\Omega}&=\frac{1}{I}\int\ide{x}(x-\xi)\times[E(x)+(v+\Omega\times(x-\xi))\times B(x)]\varphi(\xi-x)
      \end{aligned}\right.\\
&\left\{\begin{aligned}
       E(t_0)&=E_{(0)}\\
       B(t_0)&=B_{(0)}
      \end{aligned}\right.
\mspace{20mu}\left\{\begin{aligned}
        \xi(t_0)&= \xi_{(0)}\\
        v(t_0)&=v_{(0)}\\
        \Omega(t_0)&=\Omega_{(0)}\end{aligned}\right.
      \end{aligned}\quad .
\end{equation}
$E$ and $B$ satisfy as above constraints~\eqref{eq:12}, $j=e(v+\Omega\times (x-\xi))\varphi(\xi-x)$ and $\rho=e\varphi(\xi-x)$. This system of equations could be used to describe the motion of a rigid rotating charge distribution of mass $m$ and moment of inertia $I$, with angular velocity $\Omega$, coupled with its electromagnetic field \citep[see][]{Appel200124,imakomspo2004}.

Define the following conditions:
\begin{innercondition}
\label{sec:rotating-charge-1}
$\varphi$ is a differentiable function of $\mathds{R}^3$ such that $\forall i,i'=1,2,3$, $\norm{\varphi(x)}_Y$, $\norm{x_i\varphi(x)}_Y$, $\norm{x_ix_{i'}\varphi(x)}_Y <\infty$ when $Y=\dot{H}^{-s},\dot{H}^{-s+1},\dot{H}^s,\dot{H}^{s+1}$.
\end{innercondition}
\begin{innercondition}
 \label{sec:rotating-charge-2}
$\varphi$ is a differentiable function of $\mathds{R}^3$ such that $\forall i,i'=1,2,3$, $\norm{\varphi(x)}_Y$, $\norm{x_i\varphi(x)}_Y$, $\norm{x_ix_{i'}\varphi(x)}_Y <\infty$ when $Y=H^{r},H^{1}$.
\end{innercondition}
Then theorems~\ref{sec:stat-main-results-1} and~\ref{sec:stat-main-results-2} above can be reformulated for the system~\eqref{eq:15} as:
\begin{thm1*}
Let $-\infty<s<3/2$, $\varphi$ satisfying Condition~\ref{sec:rotating-charge-1} and such that~\eqref{eq:12} admits a solution on $(\dot{H}^s(\mathds{R}^3))^3\oplus (\dot{H}^s(\mathds{R}^3))^3$. Furthermore let $\xi_{(0)},v_{(0)},\Omega_{(0)}\in\mathds{R}^3$ and $E_{(0)}, B_{(0)}\in (\dot{H}^s(\mathds{R}^3))^3$ satisfying~\eqref{eq:12}. Then Cauchy problem~\eqref{eq:15} admits a unique solution on $\mathscr{C}^0(\mathds{R},(\dot{H}^s(\mathds{R}^3))^3\oplus (\dot{H}^s(\mathds{R}^3))^3\oplus\mathds{R}^3\oplus\mathds{R}^3\oplus\mathds{R}^3)$.
\end{thm1*}
\begin{thm2*}
Let $r\geq 0$, $\varphi$ satisfying Condition~\ref{sec:rotating-charge-2} and such that~\eqref{eq:12} admits a solution on $(H^r(\mathds{R}^3))^3\oplus (H^r(\mathds{R}^3))^3$. Furthermore let $\xi_{(0)},v_{(0)},\Omega_{(0)}\in\mathds{R}^3$ and $E_{(0)}, B_{(0)}\in (H^r(\mathds{R}^3))^3$ satisfying~\eqref{eq:12}. Then~\eqref{eq:15} admits a unique solution on $\mathscr{C}^0(\mathds{R},(H^r(\mathds{R}^3))^3\oplus (H^r(\mathds{R}^3))^3\oplus\mathds{R}^3\oplus\mathds{R}^3\oplus\mathds{R}^3) \cap \mathscr{C}^1(\mathds{R},$
$(H^{r-1}(\mathds{R}^3))^3\oplus (H^{r-1}(\mathds{R}^3))^3\oplus\mathds{R}^3\oplus\mathds{R}^3)$.
\end{thm2*}

\section{Local Solution.}
\label{sec:priori-cons}

In this section we construct a local solution of~\eqref{eq:6} on $\mathscr{X}_{s}(I)$, for a suitable interval $I\subset \mathds{R}$ containing $t_0$; this is done in Proposition~\ref{sec:local-solution}. We start our analysis summarising the properties of $U(t)$ and $W(t)$, defined respectively in~\eqref{eq:4} and~\eqref{eq:5}, we will use the most. This is done in the following Proposition:
\begin{proposition}\label{sec:local-solution-1}
$U(t)$ satisfies the following properties:
\begin{enumerate}[i.]
\item $U(t)$ commutes with $\omega^s$ for all $s$ and $t\in\mathds{R}$, on suitable domains.
\item $U(t)$, $t\in\mathds{R}$, is a unitary one-parameter group on $(\dot{H}^s(\mathds{R}^3))^3\oplus (\dot{H}^s(\mathds{R}^3))^3$, for $s<3/2$.
\end{enumerate}
$W(t)$ satisfies the following properties:
\begin{enumerate}[i.]
\item $W(t)$ is differentiable in $t$, and
  \begin{equation*}
    \dot{W}(t)=\left(
      \begin{array}{cc}
        0&1\\0&0
      \end{array}
\right)\; .
  \end{equation*}
\item $W(s)W(t)=W(s+t)$ for all $s,t\in\mathds{R}$.
\end{enumerate}
\end{proposition}

\subsection{Contraction mapping.}
\label{sec:contraction-mapping}

We want to prove local existence by means of Banach fixed-point theorem. In order to do that we need to define a strict contraction on a closed subspace of $\mathscr{X}_s(I)$. Define
\begin{equation*}
  \begin{split}
      A[t_0,\vec{u}_{(0)}](\vec{u}(t))\equiv\left(
    \begin{array}{cc}
      U(t-t_0)&0\\0&W(t-t_0)
    \end{array}
\right)\vec{u}_{(0)}+\int_{t_0}^t\de\tau\left(
    \begin{array}{cc}
      U(t-\tau)&0\\0& W(t-\tau)
    \end{array}
\right) \left(
    \begin{array}{c}
      -j(\tau)\\0\\0\\f_{em}(\tau)
    \end{array}
\right)\; .
  \end{split}
\end{equation*}
The following lemma is crucial for the analysis of the contraction map $A[t_0,\vec{u}_{(0)}]$:
\begin{lemma}\label{sec:contraction-mapping-1}
  Let $\vec{u_{(1)}}$ and $\vec{u_{(2)}}\in\mathscr{X}_s(I)$, with $I=[-T+t_0,t_0+T]$ for some $T>0$; $\varphi$ a differentiable function such that $\norm{\varphi}_Y<\infty$, when $Y$ is $\dot{H}^{-s}$, $\dot{H}^{-s+1}$, $\dot{H}^{s}$ and $\dot{H}^{s+1}$. Then, $\forall s<3/2$,
  \begin{equation*}
    \begin{split}
      \norm{A[t_0,\vec{u}_{(0)}](\vec{u}_{(1)})-A[t_0,\vec{u}_{(0)}](\vec{u}_{(2)})}^2_{\mathscr{X}_s(I)}\leq\ass{e}^2T^2\sup_{t\in I}\sup_{\alpha_i=1,2}\Bigl\{2\Bigl(\norm{\varphi}^2_{\dot{H}^s}\bigl\lvert(v_{(1)}-v_{(2)})(t)\bigr\rvert^2\\+3 \norm{\varphi}^2_{\dot{H}^{s+1}}\ass{v_{(\alpha_1)}(t)}^2\ass{(\xi_{(1)}-\xi_{(2)})(t)}^2 \Bigr)+2(1+T^2)\Bigl[2\Bigl(\norm{\varphi}^2_{\dot{H}^{-s+1}}\norm{F_{0(\alpha_2)}(t)}^2_{(\dot{H}^s)^3}\\\ass{(\xi_{(1)}-\xi_{(2)})(t)}^2+\norm{\varphi}^2_{\dot{H}^{-s}}\norm{(F_{0(1)}-F_{0(2)})(t)}^2_{(\dot{H}^s)^3}\Bigr)+6\Bigl(\norm{\varphi}^2_{\dot{H}^{-s}}\norm{F_{(\alpha_3)}(t)}^2_{(\dot{H}^s)^3}\\\ass{(v_{(1)}-v_{(2)})(t)}^2+\norm{\varphi}^2_{\dot{H}^{-s+1}}\ass{v_{(\alpha_4)}(t)}^2\norm{F_{(\alpha_5)}(t)}^2_{(\dot{H}^s)^3}\ass{(\xi_{(1)}-\xi_{(2)})(t)}^2+\norm{\varphi}^2_{\dot{H}^{-s}}\\\ass{v_{(\alpha_6)}(t)}^2\norm{(F_{(1)}-F_{(2)})(t)}^2_{(\dot{H}^s)^3} \Bigr) \Bigr] \Bigr\}\; .
    \end{split}
  \end{equation*}
\end{lemma}
\begin{corollary}\label{sec:contraction-mapping-2}
  If in addition to the conditions above: $\vec{u}_{(1)}$ and $\vec{u}_{(2)}\in B(I,\rho)$ (the ball of radius $\rho$ of $\mathscr{X}_s(I)$) for some $\rho>0$ and $\max_Y\norm{\varphi}_Y=M$, then
  \begin{equation*}
    \begin{split}
      \norm{A[t_0,\vec{u}_{(0)}](\vec{u}_{(1)})-A[t_0,\vec{u}_{(0)}](\vec{u}_{(2)})}_{\mathscr{X}_s(I)}\leq T(T+1)M\ass{e}(4\rho^2+5)\norm{\vec{u}_{(1)}-\vec{u}_{(2)}}_{\mathscr{X}_s(I)}\; .
    \end{split}
  \end{equation*}
\end{corollary}
\begin{corollary}\label{sec:contraction-mapping-3}
  Let $\vec{u}_{(0)}\in B(\rho)\subset \mathscr{X}_s$ (ball of radius $\rho$ of $\mathscr{X}_s$), $\vec{u}\in B(I,\rho_1)\subset \mathscr{X}_s(I)$ and $\max_Y\norm{\varphi}_Y=M$. Then
  \begin{equation*}
    \norm{A[t_0,\vec{u}_{(0)}](\vec{u})}_{\mathscr{X}_s(I)}\leq \sqrt{2}(T+1)\rho +T(T+1)M\ass{e}\rho_1(4\rho_1^2+5)\; .
  \end{equation*}
\end{corollary}
\begin{proof}[Proof of Lemma~\ref{sec:contraction-mapping-1}]
Consider the case $t_0<t$ (the other is perfectly analogous): by definition of $A[t_0,\vec{u}_{(0)}]$ and Proposition~\ref{sec:local-solution-1} we write
\begin{equation*}
  \begin{split}
    \norm{A[t_0,\vec{u}_{(0)}](\vec{u}_{(1)})-A[t_0,\vec{u}_{(0)}](\vec{u}_{(2)})}^2_{\mathscr{X}_s(I)}\leq \sup_{t\in I}\ass{t-t_0}^2 \sup_{\tau\in [t_0,t]}\Bigl\{\norm{(j_{(2)}-j_{(1)})(\tau)}^2_{(\dot{H}^s)^3}+(1\\+\ass{t-\tau}^2)\ass{(f_{em(1)}-f_{em(2)})(\tau)}^2\Bigr\}\\
\leq T^2 \ass{e}^2\sup_{t\in I}\sup_{\tau\in [t_0,t]}\Bigl\{\norm{v_{(2)}(\tau)\varphi(\xi_{(2)}(\tau)-\:\cdot\:)-v_{(1)}(\tau)\varphi(\xi_{(1)}(\tau)-\:\cdot\:)}^2_{(\dot{H}^s)^3}+2(1+\ass{t-\tau}^2)\\\Bigl[\sum_{j=1}^3\ass{(\varphi*F^{0j}_{(1)}(\tau))(\xi_{(1)}(\tau))-(\varphi*F^{0j}_{(2)}(\tau))(\xi_{(2)}(\tau))}^2 +\bigl\lvert v_{k(1)}(\varphi*F^{jk}_{(1)}(\tau))(\xi_{(1)}(\tau))\\-v_{k'(2)}(\varphi*F^{jk'}_{(2)}(\tau))(\xi_{(2)}(\tau))\bigr\rvert^2\Bigr]\Bigr\}\\
\equiv T^2 \ass{e}^2\sup_{t\in I}\sup_{\tau\in [t_0,t]}\Bigl\{X_1(\tau)+2(1+\ass{t-\tau}^2)\Bigl[X_2(\tau)+X_3(\tau) \Bigr] \Bigr\}\; .
  \end{split}
\end{equation*}
Consider now $X_i(\tau)$, $i=1,2,3$ separately.
\begin{equation*}
  \begin{split}
    X_1(\tau)\leq 2\Bigl(\ass{(v_{(2)}-v_{(1)})(\tau)}^2\norm{\varphi(\xi_{(2)}(\tau)-\:\cdot\:)}^2_{\dot{H}^s}+\ass{v_{(1)}(\tau)}^2\bigl\lVert\varphi(\xi_{(2)}(\tau)-\:\cdot\:)\\-\varphi(\xi_{(1)}(\tau-\:\cdot\:))\bigr\rVert^2_{\dot{H}^s} \Bigr)\; .
  \end{split}
\end{equation*}
Since $\varphi$ is differentiable, we can write for some $c\in[0,1]$ (by the mean value theorem):
\begin{equation*}
  \varphi(\xi_{(2)}(\tau)-x)-\varphi(\xi_{(1)}(\tau)-x)=\nabla\varphi((1-c)\xi_{(2)}(\tau)+c\,\xi_{(1)}(\tau)-x)\cdot(\xi_{(2)}(\tau)-\xi_{(1)}(\tau))\; .
\end{equation*}
Therefore we obtain
\begin{equation*}
  \begin{split}
    X_1(\tau)\leq 2\Bigl(\ass{(v_{(1)}-v_{(2)})(\tau)}^2\norm{\varphi}^2_{\dot{H}^s}+\sup_{\alpha=1,2}\ass{v_{(\alpha)}(\tau)}^2\bigl\lVert \nabla\varphi((1-c)\xi_{(2)}(\tau)\\+c\,\xi_{(1)}(\tau)-x)\cdot(\xi_{(2)}(\tau)-\xi_{(1)}(\tau))\bigr\rVert^2_{\dot{H}^s} \Bigr)\\
\leq 2\Bigl(\ass{(v_{(1)}-v_{(2)})(\tau)}^2\norm{\varphi}^2_{\dot{H}^s}+3\ass{(\xi_{(2)}-\xi_{(1)})(\tau)}^2\sup_{\alpha=1,2}\ass{v_{(\alpha)}(\tau)}^2\\\bigl\lVert \varphi((1-c)\xi_{(2)}(\tau)+c\,\xi_{(1)}(\tau)-\:\cdot\:)\bigr\rVert^2_{\dot{H}^{s+1}} \Bigr)\; ;
  \end{split}
\end{equation*}
in the last inequality we used the fact that, for all $a,b\in\mathds{R}^3$:
\begin{equation*}
  \begin{split}
    \norm{\nabla\varphi(a-x)\cdot b}^2_{\dot{H}^s}=\norm{\sum_{i=1}^3\partial_i(\omega^s\varphi(a-x))b_i}_{L^2}^2\leq 3\sum_{i=1}^3\ass{b_i}^2\norm{\partial_i(\omega^s\varphi(a-x))}_{L^2}^2\\
\leq 3\sum_{i=1}^3\ass{b_i}^2\braket{\omega^s\varphi(a-x)}{-\Delta\omega^s\varphi(a-x)}\; .
  \end{split}
\end{equation*}
Finally we obtain
\begin{equation*}
  \begin{split}
    X_1(\tau)\leq 2\Bigl(\ass{(v_{(1)}-v_{(2)})(\tau)}^2\norm{\varphi}^2_{\dot{H}^s}+3\ass{(\xi_{(1)}-\xi_{(2)})(\tau)}^2\sup_{\alpha=1,2}\ass{v_{(\alpha)}(\tau)}^2\bigl\lVert \varphi\bigr\rVert^2_{\dot{H}^{s+1}} \Bigr)\; .
  \end{split}
\end{equation*}
A similar reasoning yields the following results for $X_2(\tau)$ and $X_3(\tau)$, using the fact that $\norm{\varphi*F}_{L^\infty}\leq \norm{\omega^{-s}\varphi}_{L^2}\norm{\omega^s F}_{L^2}$:
\begin{equation*}
  \begin{split}
    X_2(\tau)\leq 2\Bigl(\ass{(\xi_{(1)}-\xi_{(2)})(\tau)}^2 \norm{\varphi}_{\dot{H}^{-s+1}}^2\sup_{\alpha=1,2}\norm{F_{0(\alpha)}(\tau)}^2_{(\dot{H}^s)^3}+\norm{\varphi}_{\dot{H}^{-s}}^2\norm{(F_{0(1)}-F_{0(2)})(\tau)}^2_{(\dot{H}^s)^3} \Bigr)\; ,
  \end{split}
\end{equation*}
\begin{equation*}
  \begin{split}
    X_3(\tau)\leq 6\Bigl(\ass{(v_{(1)}-v_{(2)})(\tau)}^2 \norm{\varphi}_{\dot{H}^{-s}}^2\sup_{\alpha=1,2}\norm{F_{(\alpha)}(\tau)}^2_{(\dot{H}^s)^3}+\norm{\varphi}_{\dot{H}^{-s}}^2\sup_{\alpha=1,2}\ass{v_{(\alpha)}(\tau)}^2\norm{(F_{(1)}\\-F_{(2)})(\tau)}^2_{(\dot{H}^s)^3}+\ass{(\xi_{(1)}-\xi_{(2)})(\tau)}^2\norm{\varphi}_{\dot{H}^{-s+1}}^2 \sup_{\alpha,\beta=1,2}\ass{v_{(\alpha)}(\tau)}^2 \norm{F_{(\beta)}(\tau)}^2_{(\dot{H}^s)^3} \Bigr)\; .
  \end{split}
\end{equation*}

\end{proof}

\subsection{Existence of local solution.}
\label{sec:exist-local-solut}

We are now able to show that, if $\vec{u}_{(0)}\in B(\rho)\subset\mathscr{X}_s$, then $A[t_0,\vec{u}_{(0)}]$ is a strict contraction of $B(I,2\rho)\subset\mathscr{X}_s(I)$ for a suitable $I$ that depends on $\rho$. This fact proves that a local solution of~\eqref{eq:6} exists on $\mathscr{X}_s(I)$. The precise statement is contained in the following proposition:
\begin{proposition}\label{sec:local-solution}
  Let $s<3/2$; $\varphi$ a differentiable function such that $\norm{\varphi}_Y\leq M$, when $Y$ is $\dot{H}^{-s}$, $\dot{H}^{-s+1}$, $\dot{H}^{s}$ and $\dot{H}^{s+1}$. Then for all $\rho> 0$, $\exists T(\rho)>0$ such that, for all $\vec{u}_{(0)}\in B(\rho)\subset \mathscr{X}_s$, Equation~\eqref{eq:6} has a unique solution belonging to $B(I,2\rho)$ with $I=[-T(\rho)+t_0,t_0+T(\rho)]$.
\end{proposition}
\begin{proof}
Corollary~\ref{sec:contraction-mapping-3} (with $\rho_1=a\rho$) shows that for all $\vec{u}_{(0)}\in B(\rho)$, $A[t_0,\vec{u}_{(0)}]$ maps $B(I,a\rho)$ into itself if $I=[-T+t_0,t_0+T]$ with $0<T\leq T_a$, $T_a(\rho)$ solution of $(T_a+1)(\sqrt{2}+T_a\ass{e}Ma(4a^2\rho^2+5))-a=0$. The last equation has at most one positive solution: the positive solution exists for all $a>\sqrt{2}$.

On the other hand, corollary~\ref{sec:contraction-mapping-2} shows that $A[t_0,\vec{u}_{(0)}]$ is a strict contraction on $B(I,2\rho)$ for $I=[-T+t_0,t_0+T]$, $0<T<T_c$, with $T_c(2\rho)$ positive solution of $T_c(T_c+1)\ass{e}M(16\rho^2+5)=1$. Therefore defining $T(\rho)>0$ as
\begin{equation*}
  T(\rho)=\min\{T_{a=2}(\rho),T_c(2\rho)/2\}\; ,
\end{equation*}
it follows that $A[t_0,\vec{u}_{(0)}]$ is a strict contraction on $B(I,2\rho)$ with $I=[-T(\rho)+t_0,t_0+T(\rho)]$. By Banach's fixed point theorem, the map $A[t_0,\vec{u}_{(0)}]$ has then a unique fixed point on $B(I,2\rho)$, solution of~\eqref{eq:6}.

\end{proof}

\section{Uniqueness, Maximal Solution.}
\label{sec:uniq-maxim-solut}

In this section we prove that for all $s<3/2$ the solution of~\eqref{eq:6} on $\mathscr{X}_s(I)$ is unique, provided it exists, for all $\vec{u}_{(0)}\in\mathscr{X}_s$ and $I\subseteq\mathds{R}$. The uniqueness result yields the possibility to construct a maximal solution, and by Proposition~\ref{sec:local-solution} we establish the finite blowup alternative.

\subsection{Uniqueness.}
\label{sec:uniqueness}

The uniqueness result is again based on Lemma~\ref{sec:contraction-mapping-1}; therefore the necessary conditions on $\varphi$ are the same.
\begin{proposition}\label{sec:uniqueness-1}
  Let $s<3/2$; $\varphi$ a differentiable function such that $\norm{\varphi}_Y\leq M$, when $Y$ is $\dot{H}^{-s}$, $\dot{H}^{-s+1}$, $\dot{H}^{s}$ and $\dot{H}^{s+1}$. Suppose that Equation~\eqref{eq:6} has at least one solution on $\mathscr{X}_s(I)$, $I\subseteq\mathds{R}$ when $\vec{u}_{(0)}\in\mathscr{X}_s$. Then the solution is unique.
\end{proposition}
\begin{proof}
  Suppose there are two solution corresponding to $\vec{u}_{(0)}\in\mathscr{X}_s$, namely $\vec{u}_{(1)}(t)$ and $\vec{u}_{(2)}(t)$. Define $\vec{u}_-=\vec{u}_{(1)}-\vec{u}_{(2)}$. Then, following a reasoning analogous to the one for the proof of Lemma~\ref{sec:contraction-mapping-1}, we obtain
  \begin{equation*}
    \norm{\vec{u}_-(t)}_{\mathscr{X}_s}\leq\ass{\int_{t_0}^t\ide{\tau}M(\tau)\norm{\vec{u}_-(\tau)}_{\mathscr{X}_s}}\; ;
  \end{equation*}
with
\begin{equation*}
  \begin{split}
    M(\tau)^2=2\ass{e}^2M^2\sup_{\alpha_i=1,2}\Bigl\{1+3\ass{v_{(\alpha_1)}(\tau)}^2+2(1+\ass{t-\tau}^2)\Bigl[2\Bigl(\norm{F_{0(\alpha_2)}(\tau)}^2_{(\dot{H}^s)^3}\\+1\Bigr)+6\Bigl(\norm{F_{(\alpha_3)}(\tau)}^2_{(\dot{H}^s)^3}+\ass{v_{(\alpha_4)}(\tau)}^2+\ass{v_{(\alpha_5)}(\tau)}^2\norm{F_{(\alpha_6)}(\tau)}^2_{(\dot{H}^s)^3}\Bigr) \Bigr] \Bigr\}\; .
  \end{split}
\end{equation*}
By Gronwall's Lemma, it follows that $\vec{u}_-=0$.

\end{proof}

\subsection{Maximal Solution.}
\label{sec:maximal-solution}

Using Propositions~\ref{sec:local-solution} and~\ref{sec:uniqueness-1} we can construct the maximal solution of~\eqref{eq:6} on $\mathscr{X}_s(\mathds{R})$. By maximal solution we mean that it is defined on the interval $I_{max}=[-T_-+t_0,t_0+T_+]$, for some $T_-,T_+>0$ and every solution on $\mathscr{X}_s(I)$ is such that $I\subseteq I_{max}$. In the next proposition we show that if either $T_-$ or $T_+$ is finite, then the $\mathscr{X}_s$ norm of the solution $\vec{u}(t)$ has to diverge when $t\to T_-$ (or $T_+$).
\begin{proposition}\label{sec:maximal-solution-1}
  Let $s<3/2$, and $T_-,T_+>0$ be such that $I_{max}=[-T_-+t_0,t_0+T_+]$ is the maximal interval where the solution of Equation~\eqref{eq:6} on $\mathscr{X}_s(I_{max})$ is defined, for $\vec{u}_0\in\mathscr{X}_s$. Let $\vec{u}_{max}$ be such solution, then one of the following is true:
  \begin{enumerate}[i.]
  \item $T_-<\infty$ and $\norm{\vec{u}_{max}(t)}_{\mathscr{X}_s}\to\infty$ when $t\to -T_-+t_0$;
  \item $T_-=\infty$.
  \end{enumerate}
Equivalently:
\begin{enumerate}[i$\,'$.]
  \item $T_+<\infty$ and $\norm{\vec{u}_{max}(t)}_{\mathscr{X}_s}\to\infty$ when $t\to t_0+ T_+$;
  \item $T_+=\infty$.
  \end{enumerate}
\end{proposition}
\begin{proof}
  Assume there is a sequence $(t_i)_{i\in\mathds{N}}$ and $N>0$ such that $t_i\to T_+$ and $\norm{\vec{u}_{max}(t_i)}_{\mathscr{X}_s}\leq N$ for all $i\in\mathds{N}$. Let $t_k$ be such that $t_k+T(N)>T_+$, where $T(N)$ is defined in Proposition~\ref{sec:local-solution}. Starting from $\vec{u}_{max}(t_k)$ one can therefore extend, by Propositions~\ref{sec:local-solution} and~\ref{sec:uniqueness-1}, the solution $\vec{u}_{max}(t)$ to $t=t_k+T(N)>T_+$. This contradicts maximality. The proof for $T_-$ is analogous. 

\end{proof}

\section{Global Existence}
\label{sec:global-existence}

In this section we prove that $I_{max}$ defined in section~\ref{sec:maximal-solution} is all $\mathds{R}$. This is done by means of an energy-type estimate, given in Lemma~\ref{sec:global-existence-1}. The result holds also when we substitute $\omega^s$ with $(1-\Delta)^{s/2}$; as stated in remark~\ref{sec:global-existence-2}.

We introduce the so called interaction representation. Define
\begin{equation*}
  \vec{F}\equiv\left(
  \begin{array}{c}
    F_0\\F
  \end{array}\right)\; .
\end{equation*}
Also, $\widetilde{\vec{F}}(t,t_0)=U^*(t-t_0)\vec{F}(t)$. Therefore if $\vec{F}(t)$ obeys the first part of~\eqref{eq:6}, we have
\begin{equation*}
  \widetilde{\vec{F}}(t)=\vec{F}_{(0)}+\int_{t_0}^t\ide{\tau}U(t_0-\tau)\left(
  \begin{array}{c}
    -j(\tau)\\0
  \end{array}\right)\; .
\end{equation*}
Then for all $s<3/2$, $\widetilde{\vec{F}}(t)$ is differentiable in $t$ on $(\dot{H}^s)^3\oplus(\dot{H}^s)^3$, if $\vec{F}_{(0)}\in (\dot{H}^s)^3\oplus(\dot{H}^s)^3$ (with $\varphi$ regular enough), and
\begin{equation}\label{eq:8}
  \partial_t \widetilde{\vec{F}}(t)=U(t_0-t) \left(\begin{array}{c}
    -j(t)\\0
  \end{array}\right)\; .
\end{equation}
We also remark that $v(t)$ satisfying \eqref{eq:6} is differentiable in $t$, and
\begin{equation}\label{eq:9}
  \dot{v}(t)=f_{em}(t)\; .
\end{equation}
\begin{lemma}\label{sec:global-existence-1}
  Let $s<3/2$, $\varphi$ such that $\norm{\varphi}_{\dot{H}^{-s}},\norm{\varphi}_{\dot{H}^{s}}< \infty$. Then the following inequality hold, for $E$, $B$ and $v$ satisfying Equation~\eqref{eq:6}:
  \begin{equation*}
    \begin{split}
      \frac{1}{2}\Bigl(v^2(t)+\int\ide{x}\bigl((\omega^{s} E(t))^2+(\omega^{s} B(t))^2\bigr) \Bigr)\leq \frac{1}{2}\Bigl(v^2(t_0)+\int\ide{x}\bigl((\omega^{s} E(t_0))^2+(\omega^{s} B(t_0))^2\bigr) \Bigr)\\\exp\Bigl\{\ass{e}\ass{t-t_0}\Bigl(1+6 \norm{\varphi}_{\dot{H}^{s}}^2+3 \norm{\varphi}_{\dot{H}^{-s}}^2 \Bigr) \Bigr\}\; .
    \end{split}
  \end{equation*}
\end{lemma}
\begin{remark}\label{sec:global-existence-4}
  If $s=0$ we can prove the conservation of energy
  \begin{equation*}
    \frac{1}{2}\Bigl(v^2(t)+\int\ide{x}(E^2(t)+B^2(t)) \Bigr)=\frac{1}{2}\Bigl(v^2(t_0)+\int\ide{x}(E^2(t_0)+B^2(t_0)) \Bigr)\; .
  \end{equation*}
\end{remark}
\begin{proof}[Proof of Lemma~\ref{sec:global-existence-1}]
  Define
  \begin{equation*}
    M(t)=\frac{1}{2}\Bigl(v^2(t)+\int\ide{x}\bigl((\omega^{s} F_0(t))^2+(\omega^{s} F(t))^2\bigr) \Bigr)\; .
  \end{equation*}
Then by Proposition~\ref{sec:local-solution-1} and~\eqref{eq:8}:
\begin{equation*}
  \begin{split}
    \dtot{M(t)}{t}=v\dot{v}+\frac{1}{2}\partial_t\braket{\omega^s \widetilde{\vec{F}}}{\omega^s \widetilde{\vec{F}}}=v\dot{v}+\braket{\omega^s \widetilde{\vec{F}}}{\omega^s U(t_0-t) \left(\begin{array}{c}
    -j(t)\\0
  \end{array}\right)}\; .
  \end{split}
\end{equation*}
Using now \eqref{eq:9}, $v\cdot v\times(\varphi*B)=0$ and again Proposition~\ref{sec:local-solution-1} we obtain:
\begin{equation*}
  \dtot{M(t)}{t}\leq \ass{e}\Bigl[\frac{v^2}{2}+\frac{3}{2}\Bigl(\norm{\varphi*E}^2_{L^\infty}+\norm{(\omega^{2s}\varphi)*E}^2_{L^\infty}+\norm{\varphi*(\omega^{2s}E)}^2_{L^\infty} \Bigr)\Bigr]\; .
\end{equation*}
Young's inequality finally yields
\begin{equation*}
  \dtot{M(t)}{t}\leq\ass{e}\Bigl(1+6 \norm{\varphi}_{\dot{H}^{s}}^2+3 \norm{\varphi}_{\dot{H}^{-s}}^2 \Bigr)M(t).
\end{equation*}
Apply Gronwall's Lemma to obtain the sought result.

\end{proof}
\begin{proposition}\label{sec:global-existence-3}
  Let $s<3/2$, $\varphi$ such that $\norm{\varphi}_{\dot{H}^{-s}},\norm{\varphi}_{\dot{H}^{s}}< \infty$. Furthermore let $\vec{u}_{(0)}\in\mathscr{X}_s$ and $\vec{u}(t)$ a solution of Equation~\eqref{eq:6} on $\mathscr{X}_s(I)$ for some $I\subseteq\mathds{R}$. Then
  \begin{equation*}
    \begin{split}
      \norm{\vec{u}(t)}^2_{\mathscr{X}_s}\leq (1+\ass{t-t_0}^2) \exp\Bigl\{\ass{e}\ass{t-t_0}\Bigl(1+6 \norm{\varphi}_{\dot{H}^{s}}^2+3 \norm{\varphi}_{\dot{H}^{-s}}^2 \Bigr) \Bigr\}\norm{\vec{u}_{(0)}}^2_{\mathscr{X}_s} .
    \end{split}
  \end{equation*}
\end{proposition}
\begin{remark}\label{sec:global-existence-2}
  Since also $(1-\Delta)^{s/2}$ commutes with $U(t)$, the following statement also holds. Let $r\geq 0$, $\varphi$ such that $\norm{\varphi}_{H^{r}}< \infty$. Also, let $\mathscr{Y}_r=(H^r)^3\oplus (H^r)^3\oplus \mathds{R}^3\oplus\mathds{R}^3$; $\vec{u}_{(0)}\in\mathscr{Y}_r$ and $\vec{u}(t)$ the solution of Equation~\eqref{eq:6} on $\mathscr{X}_0(\mathds{R})$. Then
  \begin{equation*}
    \begin{split}
      \norm{\vec{u}(t)}^2_{\mathscr{Y}_r}\leq (1+\ass{t-t_0}^2) \exp\Bigl\{\ass{e}\ass{t-t_0}\Bigl(1+9 \norm{\varphi}_{H^{r}}^2\Bigr) \Bigr\}\norm{\vec{u}_{(0)}}^2_{\mathscr{Y}_r} .
    \end{split}
  \end{equation*}
\end{remark}
\begin{proof}[Proof of Proposition~\ref{sec:global-existence-3}]
  By Equation~\eqref{eq:6} and Lemma~\ref{sec:global-existence-1} we obtain
  \begin{equation*}
    \begin{split}
      \ass{\xi(t)}^2\leq\ass{\int_{t_0}^t\ide{\tau}v(\tau)}^2\leq\ass{t-t_0}^2 \Bigl(v^2(t_0)+\int\ide{x}\bigl((\omega^{s} E(t_0))^2+(\omega^{s} B(t_0))^2\bigr) \Bigr)\\\exp\Bigl\{\ass{e}\ass{t-t_0}\Bigl(1+6 \norm{\varphi}_{\dot{H}^{s}}^2+3 \norm{\varphi}_{\dot{H}^{-s}}^2 \Bigr) \Bigr\}\; .
    \end{split}
  \end{equation*}
Hence the bound is proved using Lemma~\ref{sec:global-existence-1}.

\end{proof}
We are now able to prove that $I_{max}$, defined on section~\ref{sec:maximal-solution} is $\mathds{R}$:
\begin{proposition}\label{sec:global-existence-5}
  Let $s<3/2$, $\varphi$ a differentiable function such that $\norm{\varphi}_Y<\infty$, when $Y$ is $\dot{H}^{-s}$, $\dot{H}^{-s+1}$, $\dot{H}^{s}$ and $\dot{H}^{s+1}$. Furthermore let $\vec{u}_{(0)}\in\mathscr{X}_s$ and $\vec{u}_{max}\in\mathscr{X}_s(I_{max})$ the corresponding maximal solution of Equation~\eqref{eq:6}. Then $I_{max}=\mathds{R}$.
\end{proposition}
\begin{proof}
  By Proposition~\ref{sec:global-existence-3} we see that, for all $s<3/2$, $\norm{\vec{u}_{max}(t)}_{\mathscr{X}_s}$ diverges if and only if $t\to\pm\infty$. Then, by Proposition~\ref{sec:maximal-solution-1}, $I_{max}=\mathds{R}$. 

\end{proof}

\section{The Second Couple of Maxwell's Equations.}
\label{sec:second-couple-maxw}

In this section we analyse the divergence Maxwell's equations, $\nabla\cdot E=\rho$, $\nabla\cdot B=0$. We prove that any couple of tempered distributions that satisfies~\eqref{eq:11} (as distributions, i.e. acting on $\mathscr{S}(\mathds{R}^3)$ functions), satisfies also~\eqref{eq:12}, provided the initial data satisfy~\eqref{eq:12} themselves. Therefore the divergence part of Maxwell's equations reduces to a constraint on the set of possible initial values of the electric and magnetic field.

\begin{proposition}\label{sec:second-couple-maxw-1}
  Let $E_{(0)}$, $B_{(0)}$ satisfy~\eqref{eq:12}; $E(t)$, $B(t)\in \mathscr{S}'(\mathds{R}^3)$ satisfy Cauchy problem~\eqref{eq:11}. Also, let charge conservation~\eqref{eq:2} holds. Then $E(t)$, $B(t)$ satisfy~\eqref{eq:12}. 
\end{proposition}
\begin{remark}\label{sec:second-couple-maxw-2}
  To apply this proposition, we need initial data of integral Equation~\eqref{eq:6} that satisfy~\eqref{eq:12}. So given $\varphi$, we have to find vectors in $(\dot{H}^s)^3$, $s<3/2$, that satisfy Equations~\eqref{eq:12}. For example $\{B\in (\dot{H}^s)^3:\,\nabla\cdot B=0\}$ is a closed subspace of $(\dot{H}^s)^3$, whose orthogonal complement is $\{B\in (\dot{H}^s)^3:\, \exists \Lambda\in \dot{H}^{s+1}\text{ such that } B=\nabla\Lambda\}$.

A thorough study of these equations is, however, beyond the scope of this paper, the interested reader should refer to \citet[][and references thereof contained]{csato2011pullback}. Keep in mind that in general to fullfil~\eqref{eq:12} some regularity conditions on $\varphi$ may be necessary.
\end{remark}
\begin{proof}[Proof of Proposition~\ref{sec:second-couple-maxw-1}]
Define the distributions $f(t)$, $g(t)\in\mathscr{S}'(\mathds{R}^3)$ as
\begin{align*}
  f(t)&=\nabla\cdot E(t) -\rho(t)\; ; \\
  g(t)&=\nabla\cdot B(t)\; .
\end{align*}
Using the assumptions on initial data we see that $f(t_0)=g(t_0)=0$. Furthermore, in the sense of distributions, using~\eqref{eq:11} we obtain:
\begin{align*}
  \partial_t f(t)&=\nabla\cdot \partial_t E(t) -\partial_t\rho(t)=\nabla\cdot\nabla\times B(t)-\nabla\cdot j(t)-\partial_t\rho(t)\; ; \\
  \partial_t g(t)&=\nabla\cdot \partial_t B(t)=-\nabla\cdot\nabla\times E(t)\; .
\end{align*}
Now since the divergence of a curl is equal to zero and using charge conservation~\eqref{eq:2}:
\begin{align*}
  \partial_t f(t)&=0\; ; \\
  \partial_t g(t)&=0\; .
\end{align*}
Therefore $f(t)=g(t)=0$ on $\mathscr{S}'(\mathds{R}^3)$.

\end{proof}


\begin{thebibliography}{15}
\providecommand{\natexlab}[1]{#1}
\providecommand{\url}[1]{\texttt{#1}}
\expandafter\ifx\csname urlstyle\endcsname\relax
  \providecommand{\doi}[1]{doi: #1}\else
  \providecommand{\doi}{doi: \begingroup \urlstyle{rm}\Url}\fi

\bibitem[Appel and Kiessling(2001)]{Appel200124}
W.~Appel and M.~K.-H. Kiessling.
\newblock Mass and spin renormalization in lorentz electrodynamics.
\newblock \emph{Annals of Physics}, 289\penalty0 (1):\penalty0 24 -- 83, 2001.
\newblock ISSN 0003-4916.
\newblock \doi{http://dx.doi.org/10.1006/aphy.2000.6119}.
\newblock URL
  \url{http://www.sciencedirect.com/science/article/pii/S0003491600961190}.

\bibitem[Bahouri et~al.(2011)Bahouri, Chemin, and Danchin]{MR2768550}
H.~Bahouri, J.-Y. Chemin, and R.~Danchin.
\newblock \emph{Fourier analysis and nonlinear partial differential equations},
  volume 343 of \emph{Grundlehren der Mathematischen Wissenschaften
  [Fundamental Principles of Mathematical Sciences]}.
\newblock Springer, Heidelberg, 2011.
\newblock ISBN 978-3-642-16829-1.
\newblock \doi{10.1007/978-3-642-16830-7}.
\newblock URL \url{http://dx.doi.org/10.1007/978-3-642-16830-7}.

\bibitem[Bambusi and Noja(1996)]{bambusi.noja}
D.~Bambusi and D.~Noja.
\newblock On classical electrodynamics of point particles and mass
  renormalization: Some preliminary results.
\newblock \emph{Letters in Mathematical Physics}, 37\penalty0 (4):\penalty0
  449--460, 1996.
\newblock ISSN 0377-9017.
\newblock \doi{10.1007/BF00312675}.
\newblock URL \url{http://dx.doi.org/10.1007/BF00312675}.

\bibitem[Bauer and D{\"u}rr(2001)]{bauerdurr2001}
G.~Bauer and D.~D{\"u}rr.
\newblock The maxwell-lorentz system of a rigid charge.
\newblock \emph{Annales Henri Poincar{\'e}}, 2\penalty0 (1):\penalty0 179--196,
  2001.
\newblock ISSN 1424-0637.
\newblock \doi{10.1007/PL00001030}.
\newblock URL \url{http://dx.doi.org/10.1007/PL00001030}.

\bibitem[Bauer et~al.(2013{\natexlab{a}})Bauer, Deckert, and
  D{\"u}rr]{MR3085905}
G.~Bauer, D.-A. Deckert, and D.~D{\"u}rr.
\newblock On the existence of dynamics in {W}heeler-{F}eynman electromagnetism.
\newblock \emph{Z. Angew. Math. Phys.}, 64\penalty0 (4):\penalty0 1087--1124,
  2013{\natexlab{a}}.

\bibitem[Bauer et~al.(2013{\natexlab{b}})Bauer, Deckert, and
  D{\"u}rr]{MR3169754}
G.~Bauer, D.-A. Deckert, and D.~D{\"u}rr.
\newblock Maxwell-{L}orentz dynamics of rigid charges.
\newblock \emph{Comm. Partial Differential Equations}, 38\penalty0
  (9):\penalty0 1519--1538, 2013{\natexlab{b}}.

\bibitem[Csat{\'o} et~al.(2011)Csat{\'o}, Dacorogna, and
  Kneuss]{csato2011pullback}
G.~Csat{\'o}, B.~Dacorogna, and O.~Kneuss.
\newblock \emph{The Pullback Equation for Differential Forms}.
\newblock Progress in nonlinear differential equations and their applications.
  Springer, 2011.
\newblock ISBN 9780817683139.

\bibitem[{Ginibre} and {Velo}(1981)]{zbMATH03764412}
J.~{Ginibre} and G.~{Velo}.
\newblock {The Cauchy problem for coupled Yang-Mills and scalar fields in the
  temporal gauge.}
\newblock \emph{{Commun. Math. Phys.}}, 82:\penalty0 1--28, 1981.
\newblock ISSN 0010-3616; 1432-0916/e.
\newblock \doi{10.1007/BF01206943}.

\bibitem[Imaykin et~al.(2004{\natexlab{a}})Imaykin, Komech, and
  Mauser]{mauser.imaikin}
V.~Imaykin, A.~Komech, and N.~Mauser.
\newblock Soliton-type asymptotics for the coupled maxwell-lorentz equations.
\newblock \emph{Annales Henri Poincar{\'e}}, 5\penalty0 (6):\penalty0
  1117--1135, 2004{\natexlab{a}}.
\newblock ISSN 1424-0637.
\newblock \doi{10.1007/s00023-004-0193-5}.
\newblock URL \url{http://dx.doi.org/10.1007/s00023-004-0193-5}.

\bibitem[Imaykin et~al.(2004{\natexlab{b}})Imaykin, Komech, and
  Spohn]{imakomspo2004}
V.~Imaykin, A.~Komech, and H.~Spohn.
\newblock Rotating charge coupled to the maxwell field: Scattering theory and
  adiabatic limit.
\newblock \emph{Monatshefte f{\"u}r Mathematik}, 142\penalty0 (1-2):\penalty0
  143--156, 2004{\natexlab{b}}.
\newblock ISSN 0026-9255.
\newblock \doi{10.1007/s00605-004-0232-9}.
\newblock URL \url{http://dx.doi.org/10.1007/s00605-004-0232-9}.

\bibitem[Imaykin et~al.(2011)Imaykin, Komech, and Spohn]{imaykin:042701}
V.~Imaykin, A.~Komech, and H.~Spohn.
\newblock Scattering asymptotics for a charged particle coupled to the maxwell
  field.
\newblock \emph{Journal of Mathematical Physics}, 52\penalty0 (4):\penalty0
  042701, 2011.
\newblock \doi{10.1063/1.3567957}.
\newblock URL \url{http://link.aip.org/link/?JMP/52/042701/1}.

\bibitem[Kiessling(1999)]{Kiessling1999197}
M.~K.-H. Kiessling.
\newblock Classical electron theory and conservation laws.
\newblock \emph{Physics Letters A}, 258\penalty0 (4--6):\penalty0 197 -- 204,
  1999.
\newblock ISSN 0375-9601.
\newblock \doi{http://dx.doi.org/10.1016/S0375-9601(99)00340-0}.
\newblock URL
  \url{http://www.sciencedirect.com/science/article/pii/S0375960199003400}.

\bibitem[Komech and Spohn(2000)]{doi:10.1080/03605300008821524}
A.~Komech and H.~Spohn.
\newblock Long---time asymptotics for the coupled maxwell---lorentz equations.
\newblock \emph{Communications in Partial Differential Equations}, 25\penalty0
  (3-4):\penalty0 559--584, 2000.
\newblock \doi{10.1080/03605300008821524}.
\newblock URL
  \url{http://www.tandfonline.com/doi/abs/10.1080/03605300008821524}.

\bibitem[Rohrlich(2007)]{rohrlich2007classical}
F.~Rohrlich.
\newblock \emph{Classical Charged Particles}.
\newblock World Scientific, 2007.
\newblock ISBN 9789812700049.

\bibitem[Spohn(2004)]{MR2097788}
H.~Spohn.
\newblock \emph{Dynamics of charged particles and their radiation field}.
\newblock Cambridge University Press, Cambridge, 2004.
\newblock ISBN 0-521-83697-2.
\newblock \doi{10.1017/CBO9780511535178}.
\newblock URL \url{http://dx.doi.org/10.1017/CBO9780511535178}.

\end{thebibliography}

\end{document}